\documentclass{amsart}

\usepackage{amsmath}
\usepackage{amsfonts}
\usepackage{amstext}
\usepackage{amsbsy}
\usepackage{amsopn}
\usepackage{amsxtra}
\usepackage{upref}
\usepackage{amsthm}
\usepackage{amsmath}
\usepackage{amssymb}
\usepackage{epsfig,verbatim,ifthen,graphicx}
\newtheorem{prop}{Proposition}[section]
\newtheorem{lema}{Lemma}[section]
\newtheorem{teo}{Theorem}[section]

\newtheorem{coro}{Corollary}[section]

\theoremstyle{definition}
\newtheorem{defi}{Definition}[section]
\newtheorem{rem}{Remark}[section]

\def\R{{\mathbb R}}

\def\N{{\mathbb N}}

\def\F{{\mathcal F}}
\def\M{{\mathcal M}}

\def\A{{\mathcal A}}

\title[Zero temperature limits of Gibbs  states]{
Zero temperature limits of Gibbs  states for almost-additive potentials}
\date{\today}

\begin{thanks}
{We would like to thank Ian Morris for his useful and interesting comments that improved the paper. Both authors were supported by  the Center of Dynamical Systems and Related Fields c\'odigo ACT1103  PIA - Conicyt. GI was partially  supported by Proyecto Fondecyt 1110040 and YY was supported by Proyecto Fondecyt Iniciaci\'{o}n 11110543. }
\end{thanks}


\author{Godofredo Iommi} \address{Facultad de Matem\'aticas,
Pontificia Universidad Cat\'olica de Chile (PUC), Avenida Vicu\~na Mackenna 4860, Santiago, Chile}
\email{giommi@mat.puc.cl}
\urladdr{http://www.mat.puc.cl/\textasciitilde giommi/}
\author{Yuki Yayama}
\address{Departamento de Ciencias B\'{a}sicas, Universidad del B\'{i}o-B\'{i}o, Avenida Andr\'{e}s Bello, s/n
Casilla 447, Chill\'{a}n, Chile}
\email{yyayama@ubiobio.cl}

\begin{document}

\begin{abstract}
This paper is devoted to study ergodic optimisation problems for almost-additive sequences of functions (rather than a fixed potential) defined over countable Markov shifts (that is a non-compact space). Under certain assumptions we prove that any accumulation point of a family of Gibbs equilibrium states is a maximising measure. Applications are given in the study of the joint spectral radius and to multifractal analysis of Lyapunov exponent of non-conformal maps.\end{abstract}

\maketitle

\section{Introduction}
In statistical mechanics a very important problem is that of describing how do
Gibbs states varies  as the temperature changes. Of particular importance is the case when the
temperature decreases to zero.  Indeed, this case is related to ground states, that is, measures supported on configurations of minimal energy \cite[Appendix B.2]{efs}. It turns out that materials at low temperature tend to be  highly ordered, they might even reach
crystal or quasi crystal configurations.  Grounds states are the measures that account for this phenomena (see \cite[Chapter 3]{bll} for details).  A similar problem, in the context of dynamical systems, has been the subject of great interest over the last years. Indeed, given a dynamical system $(\Sigma, \sigma)$ and an observable $\phi:\Sigma \toÊ\R$ we say that a $\sigma$-invariant measure $\mu$  is a \emph{maximising} measure for $\phi$ if
\begin{equation*}
\int \phi \ d \mu = \sup \left\{ \int \phi \ d \nu : \nu \in \M \right\},
\end{equation*}
where $\M$ denotes the set of $\sigma-$invariant probability measures. In certain cases, some maximising measures can be described as the limit of  Gibbs states as the temperature goes to zero. Indeed, assume that $(\Sigma, \sigma)$ is a transitive sub-shift of finite type defined over a finite alphabet and that $\phi$ is a H\"older potential. It is well known that for every $t\in \R$ there exists a unique Gibbs state $\mu_t$ (which is also an equilibrium measure) for the potential $t \phi$ (see \cite[Theorem 1.2 and 1.22]{bow}). It turns out that if $\mu$ is any
weak-star accumulation point of $\{\mu_t\}_{t> 0}$ then $\mu$ is a maximsing measure (see, for example, \cite[Section 4]{j1}). Note that the value $t$ can be thought of as the inverse of the temperature, hence as the temperature decreases to zero the value of $t$ tends to infinity. The theory that studies maximising measures is usually called \emph{ergodic optimisation}. See \cite{bo,j1,j2} for more details.

The purpose of the present paper is to study a similar problem in the context of countable Markov shifts and for sequences of potentials. To be more precise, let $(\Sigma, \sigma)$ be a countable Markov shift  and let
$\F=\{\log f_n\}_{n=1}^{\infty}$ an almost-additive sequence of continuous functions $\log f_n : \Sigma \to \R$ (see Section \ref{sec:backg} for precise definitions). We say that a $\sigma$-invariant measure $\mu$  is a \emph{maximising} measure for $\F$ if
\begin{equation*}
 \alpha(\F):=  \sup \left\{\lim_{n\rightarrow \infty}\frac{1}{n}\int \log f_n \ d\nu : \nu \in \M \right\}= \lim_{n\rightarrow \infty}\frac{1}{n}\int \log f_n \ d\mu.
\end{equation*}
In our main result, Theorem \ref{main4}, we prove that under certain assumptions there exists a sequence of Gibbs states
$\{\mu_t\}_{t\geq 1}$  corresponding to $t \F$ and that this sequences has an accumulation point $\mu$. Moreover,  the measure $\mu$ is maximising for $\F$.  We stress that since the space $\Sigma$ is not compact the space $\mathcal{M}$ is not compact either. Therefore, the existence of an accumulation point is far from trivial. A similar problem, in the case of a single function instead of a sequence $\F$, was first studied by Jenkinson, Mauldin and Urba\'nski \cite{jmu} (see also \cite{big,bf,i,ke}).

Note that even though we prove the existence of an accumulation point for the sequence $\{\mu_t\}_{t\geq 1}$ , this does not imply that the limit $\lim_{t \to \infty} \mu_t$ exists. Actually, in the simpler setting of compact sub-shifts of finite type and H\"older potentials,  Hochman and Chazottes \cite{hc} constructed an example where there is no convergence. However,  under certain finite range assumptions convergence has been proved in \cite{br,l, chgu}.

The thermodynamic formalism needed in the context of almost-additive sequences for countable Markov shifts was developed in \cite{iy}. In particular, the existence of Gibbs states was established in \cite[Theorem 4.1]{iy}.  New results in this direction are obtained in Section \ref{sec:gibbs}, where we establish conditions that ensure that certain Gibbs states are actually equilibrium measures (recall that in this non-compact setting this is not always the case, see \cite[p.1757]{s3}).

We stress that our formalism allows us to deal with products of matrices and it is well suited for applications.
In particular, in Section \ref{sec:joint}, applications of our results are given to the study of the joint spectral
radius of a set of matrices. We construct an invariant measure that realises the joint spectral radius. Moreover, we construct a sequence of Gibbs states that can be used to approximate the value of the joint spectral radius. It should be pointed out that these results are new even in the case of finitely many matrices. In the same spirit,
Section \ref{sec:singular} is devoted to another application of our results in the setting of multifractal analysis.
For non-conformal dynamical systems defined on the plane, we  obtain the upper bound for the Lyapunov spectrum.
Moreover, we construct a measure supported on the maximal level set.

\section{Preliminaries} \label{sec:backg}
In this section, we give a brief overview of recent results of thermodynamic formalism for
almost-additive sequences on countable Markov shifts. We collect results mostly from \cite{iy}.

Let $(\Sigma, \sigma)$ be a one-sided Markov shift
over a countable alphabet $S$. This means that there exists a matrix
$(t_{ij})_{S \times S}$ of zeros and ones (with no row and no column
made entirely of zeros) such that
\begin{equation*}
\Sigma=\left\{ x\in S^{\N} : t_{x_{i} x_{i+1}}=1 \ \text{for every $i
\in \N$}\right\}.
\end{equation*}
The \emph{shift map} $\sigma:\Sigma \to \Sigma$ is defined by $\sigma(x_1 x_2 \dots)=(x_2 x_3 \dots)$.  Sometimes we simply say that $(\Sigma, \sigma)$
is a \emph{countable Markov shift}.  The set
\begin{equation*}
 C_{i_1 \cdots i_{n}}= \left\{x \in \Sigma : x_j=i_j \text{ for } 1 \le j \le n \right\}
 \end{equation*}
 is called \emph{cylinder} of length $n$. The space $\Sigma$ endowed with the topology generated by cylinder sets is a non-compact space.  We denote by $\M$ the set of $\sigma$-invariant Borel probability measures on $\Sigma$. We will always assume $(\Sigma, \sigma)$ to be topologically mixing, that is, for every $a,b \in S$ there exists $N_{ab} \in \N$ such that
for every $n > N_{ab}$ we have $C_a \cap \sigma^{-n} C_b \neq \emptyset$.

\begin{defi} \label{aaa}
Let $(\Sigma,\sigma)$ be a one-sided countable state Markov shift. For each $n\in \N$, let  $f_{n}: \Sigma \to \R^{+}$ be a continuous function.
A sequence $\mathcal{F}= \{ \log f_n \}_{n=1}^{\infty}$ on $\Sigma$ is called \emph{almost-additive} if there exists a constant $C\geq0$ such that for every $n,m\in \N, x\in \Sigma$, we have
\begin{equation} \label{A1}
 f_n(x) f_{m}(\sigma^n x) e^{-C} \leq f_{n+m}(x),
 \end{equation}
and
\begin{equation} \label{A2}
 f_{n+m}(x) \leq  f_n(x) f_{m}(\sigma^n x) e^{C}.
    \end{equation}
   \end{defi}

Throughout  this paper, we will assume the sequence $\F$ to be almost-additive.
We also assume the following regularity  condition.

\begin{defi}\label{bowen}
Let $(\Sigma, \sigma)$ be a one-sided countable Markov shift. For each $n\in \N$, let $f_{n}: \Sigma \rightarrow \R^{+}$ be continuous. A sequence
$\mathcal{F}= \{ \log f_n \}_{n=1}^{\infty}$ on $\Sigma$ is called a \emph{Bowen}
sequence if there exists $M \in \R^{+}$ such that
\begin{equation}\label{bowenbound}
 \sup \{ A_n : n \in \NÊ\} \leq M,
\end{equation}
where
\[A_n= \sup \left\{ \frac{f_n(x)}{f_n(y)} : x,y  \in \Sigma, x_i=y_i \textrm{ for } 1 \leq i \leq n\right\}.\]
\end{defi}

In \cite{iy} thermodynamic formalism was developed for almost-additive Bowen sequences. The following definition of pressure is a generalisation of the one given by Sarig \cite{s1} to the case of  almost-additive sequences.

\begin{defi} \label{def:pre} Let $\F=\{\log f_n\}_{n \in \N}$ be an almost-additive Bowen sequence, the
\emph{Gurevich pressure} of $\F$, denoted by $P(\F)$, is defined by
$$P(\F)=\lim_{n\rightarrow \infty} \frac{1}{n} \log \left(\sum_{\sigma^{n}x=x}f_n(x)\chi_{C_a}(x)\right),$$
where $\chi_{C_a}(x)$ is the characteristic function of the cylinder $C_a$.
\end{defi}
Let us stress that the limit always exists and does not depend on the choice of the cylinder.  Note that  if $f:\Sigma \rightarrow \R$ is a continuous function, the sequence of Birkhoff sums of $f$
forms an almost additive sequence (in this case the constant $C$ in definition \ref{aaa} is equal to zero). For every $n \in \N, x\in \Sigma$, define $f_n:\Sigma\rightarrow \R^{+}$
by $f_n(x)=e^{f(x)+f(\sigma x) +\dots+f (\sigma^{n-1}x)}$. Then the sequence $\{\log f_n\}_{n=1}^{\infty}$ is additive.  This remark is the link that ties up the thermodynamic formalism for a continuous function with that of sequences of continuous functions. Therefore,  the definition of pressure given in definition \ref{def:pre} generalises that of Gurevich pressure given by Sarig \cite{s1}. Also, we note that
$\lim_{n \rightarrow \infty}\frac{1}{n} \int \log f_n \  d\mu= \int f \  d\mu$ for any $\mu\in \M$.
Therefore, the next theorem is a
generalisation of the variational principle for continuous functions to the setting of almost-additive sequence of continuous functions. It was proved in \cite[Theorem 3.1]{iy}. In order to state it we need the following definition, given $f: \Sigma \to \R$ a continuous function, the \emph{transfer operator} $L_{f}$   applied to  function $g: \Sigma \rightarrow \R$
is formally defined  by
\begin{equation*} \label{transfer}
\left( L_{f} g \right) (x) := \sum_{\sigma z=x} f(z) g(z) \quad \text{ for every } x\in \Sigma.
\end{equation*}

\begin{teo}\label{main1}
Let $(\Sigma, \sigma)$ be a topologically mixing countable state Markov shift and $\F=\{ \log f_n \}_{n=1}^{\infty}$ be an almost-additive
Bowen sequence on $\Sigma$ with  $\vert \vert L_{f_1}1 \vert \vert _{\infty}<\infty$.
Then $-\infty<P(\F)<\infty$ and
\begin{align*}
P(\F)&=\sup \left\{ h(\mu) + \lim_{n\to \infty} \frac{1}{n} \int \log f_n \  d\mu : \mu \in \mathcal{M} \textrm{ and }
\lim_{n \rightarrow \infty}\frac{1}{n} \int \log f_n \  d\mu \neq -\infty
\right\}\\
&=\sup \left\{ h(\mu) + \int \lim_{n\to \infty}\frac{1}{n}\log f_n \  d\mu : \mu \in \mathcal{M} \textrm{ and }
\int \lim_{n \rightarrow \infty}\frac{1}{n}\log f_n \  d\mu \neq -\infty
\right\}.
\end{align*}
\end{teo}
A measure $\mu \in \M$ is said to be an
\emph{equilibrium measure} for $\F$ if
\begin{equation*}
P(\F)= h(\mu) + \lim_{n \to \infty} \frac{1}{n} \int \log f_n \ d \mu.
\end{equation*}

In \cite{b2,m}, the notion of Gibbs state for continuous functions was extended to the case of almost-additive sequences.

\begin{defi} \label{def-gibbs}
Let  $(\Sigma, \sigma)$ be a topologically mixing countable state Markov shift  and
$\F=\{\log f_n\}_{n=1}^{\infty}$ be an almost-additive  sequence on $\Sigma$. A measure $\mu \in \M$ is said to
be a \emph{Gibbs} state for $\F$ if there exist constants $C_{0}>0$ and $P \in \R$ such that for every
$n \in \N$ and every $x \in C_{i_1 \dots i_{n}}$ we have
\begin{equation}\label{gibbs}
\frac{1}{C_{0}} \leq	\frac{\mu(C_{i_1 \dots i_{n}})}{\exp(-nP)f_n(x)}	\leq C_{0}.
\end{equation}
 \end{defi}
In the case of a single continuous function defined over a countable Markov shift, there exists a combinatorial obstruction on $\Sigma$ that
prevents the  existence  of Gibbs states (see \cite{s3}). This, of course, is also the
case in the setting of almost additive sequences. The combinatorial condition on $\Sigma$ is the following.
\begin{defi} \label{BIP}
A countable Markov shift  $(\Sigma, \sigma)$ is said to satisfy  the \emph{big images and preimages property (BIP property)} if
there exists $\{ b_{1} , b_{2}, \dots, b_{n} \}$ in the alphabet $S$ such that
\begin{equation*}
\forall a \in S \textrm{ } \exists i,j \textrm{ such that } t_{b_{i}a}t_{ab_{j}}=1.
\end{equation*}
\end{defi}

In \cite[Theorem 4.1]{iy}, the existence of Gibbs states for an almost-additive sequence
of continuous functions defined over a BIP shift $\Sigma$ was established. Moreover, under a finite entropy
assumption it was shown that this Gibbs state is also an equilibrium measure.

\begin{teo}\label{main2}
Let $(\Sigma, \sigma)$ be a topologically mixing countable state Markov shift with the BIP property. Let $\F=\{\log f_n\}_{n=1}^{\infty}$ be an almost-additive Bowen
sequence defined on $\Sigma$. Assume that $\sum_{a\in S}\sup f_1\vert _{C_a}<\infty$.
Then there exits a unique invariant Gibbs state $\mu$ for $\F$ and it is mixing. Moreover, If $h(\mu)<\infty$, then
it is the unique equilibrium measure for $\F$.
\end{teo}

\begin{rem} \label{rteo1} Note that $\sum_{a\in S}\sup f_1\vert _{C_a}<\infty$ implies that $-\infty<P(\F)<\infty$.
\end{rem}

\section{Existence of Gibbs equilibrium states} \label{sec:gibbs}
In Theorem \ref{main2} we established conditions under which the existence of a Gibbs
state $\mu$ for an almost-additive sequence   $\F=\{\log f_n\}_{n=1}^{\infty}$ is guaranteed. It might happen that $h(\mu)=\infty$ and that $\lim_{n\to \infty} \frac{1}{n} \int \log f_n \  d\mu = -\infty$. In this case, since the sum of these two quantities is meaningless, we don't say that the measure $\mu$ is an equilibrium measure. However, if
$h(\mu) < \infty$ then  $\mu$ is indeed an equilibrium measure.
The purpose of the following section is to establish other conditions that would also imply that the Gibbs state $\mu$ is
an equilibrium measure. Our result could be compared with those in \cite{mu}, where a single function (instead of
an almost-additive sequence) is studied.

In the rest of the paper, we identify a countable alphabet $S$ with $\N$.
%


\begin{prop}\label{chara}
Let $(\Sigma, \sigma)$ be a topologically mixing countable state Markov shift with the BIP property. Let $\F=\{\log f_n\}_{n=1}^{\infty}$ be an almost-additive Bowen defined on $\Sigma$ satisfying
$\sum_{i\in \N} \sup f_1\vert_{C_{i}}<
\infty$ and   $\mu_{\F}$ be the unique invariant Gibbs state for $\F$.
The followings statements are equivalent:
\begin{enumerate}
\item $h_{\mu_{\F}}(\sigma)<\infty$, \label{1}
\item  $\lim_{n \rightarrow \infty}\frac{1}{n}\int \log f_n d\mu_{\F}>-\infty$, \label{2}
\item $\int \log f_1 d\mu_{\F}>-\infty$, \label{3}
\item $\sum_{i=1}^{\infty}\sup \{\log f_1(x):x\in C_{i}\}\sup \{f_1(x):x\in C_{i}\}>-\infty$.\label{4}
\end{enumerate}
Therefore, if one of the above is satisfied, then $\mu_{\F}$ is the unique Gibbs equilibrium state for $\F$.
\end{prop}

\begin{proof}
Suppose we have (\ref{1}). Then it follows from Theorem \ref{main2}
that $\mu_{\F}$ is the unique Gibbs equilibrium state for $\F$. Thus (1) implies (2).
Assume now (\ref{2}). Since $\sum_{i\in \N} \sup f_1\vert_{C_i}<
\infty$ implies $-\infty <P(\F)<\infty$, it is a direct consequence of
the variational principle that $h_{\mu_{\F}}(\sigma)<\infty$. Thus (2) implies (1).
Next we show that  (\ref{2}) if and only if (\ref{3}). It directly follows from Definition \ref{aaa} (equations  (\ref{A1}) and (\ref{A2})), that
\begin{equation*}
-\frac{n-1}{n}C+\int \log f_{1}(x)d\mu_{\F}\leq \frac{1}{n}\int \log f_n d\mu_{\F}\leq \frac{n-1}{n}C+ \int \log f_1
d\mu_{\F}.
\end{equation*}
Letting $n\rightarrow\infty$, we obtain the result. Now we show that (\ref{3}) if and only if (\ref{4}).
Fix $i\in \N$. Since $\F$ is a  Bowen sequence on $\Sigma$, there exists $M$ such that
\begin{equation}\label{bounv}
\sup\left\{\frac{f_1(x)}{f_1(y)}:x_1=y_1=i \right \}\leq M.
\end{equation}
This implies that
$-\log M\leq \log f_{1}(x)-\log f_{1}(y)\leq \log M$ for any $x, y \in C_i$. Therefore,
\begin{equation*}
 \sup\left\{\log \frac{f_{1}(x)}{M}:x\in C_i \right\}\leq \inf \left\{\log f_1(x):x\in C_i \right\}.
\end{equation*}
Let $x_{i}$ be an arbitrary point in $C_i$ , then
\begin{align*}
\int \log f_1 d\mu_{\F}&= \sum_{i=1}^{\infty} \int_{C_i} \log f_1 d\mu_{\F}\\
&\geq \sum_{i=1}^{\infty} \inf \left\{\log f_1(x):x\in C_i \right\}e^{-P(\F)}\frac{f_1(x_i)}{C_0}\\
&\geq \sum_{i=1}^{\infty} \sup \left\{\log \frac{f_1(x)}{M}:x\in C_i \right\} e^{-P(\F)}\frac{f_1(x_i)}{C_0},
\end{align*}
for some $C_{0}>0$, where in the second inequality we used the definition of Gibbs sate. Using (\ref{bounv}), we have
\begin{equation*}
\int \log f_1 d\mu_{\F}\geq \frac{e^{-P(\F)}}{C_0 M}\sum_{i=1}^{\infty}
\sup \left\{\log \frac{f_1(x)}{M}:x\in C_i \right\} \sup\{f_1(x):x\in C_i\}.
\end{equation*}
Therefore, (\ref{4}) implies (\ref{3}).
To see that (\ref{3}) implies (\ref{4}), consider
\begin{align*}
 \int \log f_1 d\mu_{\F}&=\sum_{i=1}^{\infty}\int _{C_i}\log f_1 d\mu_{\F}\\
&\leq e^{-P(\F)}C_0\sum_{i=1}\sup \{\log f_1 (x):x\in C_i\}f_1(x_i)\\
&\leq e^{-P(\F)}C_0\sum_{i=1}\sup \{\log f_1(x):x\in C_i\}\sup\{f_1(x):x\in C_i\}.
\end{align*}
Therefore, the desired result is obtained.
\end{proof}

The rest of the section is devoted to  study the relation between the Gibbs equilibrium state for $\log f_1$ on $\Sigma$ and
that for $\F=\{\log f_n\}_{n=1}^{\infty}$ on $\Sigma$.
We begin establishing conditions on a continuous function $\log f_1$ so that it has a Gibbs equilibrium state (see \cite{mu, s1,s3} for related work).

\begin{prop}\label{general}
Let $(\Sigma, \sigma)$ be a topologically mixing countable Markov shift with the BIP property. Suppose that
$\log f_1$ is a continuous function on $\Sigma$ satisfying:
\begin{enumerate}
\item $\sup_{n\in \N}\left\{\frac{f_1(x)f_1(\sigma x)\dots
f_1(\sigma^{n-1}x)}{f_1(y)f_1(\sigma x)\dots
f_1(\sigma^{n-1}y)}: x_{i}=y_{i}, 1\leq i\leq n \right\}<\infty$, \label{pp1}
\item  $\sum_{i\in \N} e^{\sup\log f_1\vert_{C_i}}<\infty$. \label{pp2}
\end{enumerate}
Then there exists a unique invariant Gibbs state  $\mu_{\log f_1}$ for $\log f_1$.
Moreover,
\begin{enumerate}
\item If $\int \log f_1 d\mu_{\log f_1}>-\infty$, then
$\mu_{\log f_1}$ is the unique Gibbs equilibrium state for $\log f_1$.
\item $\int \log f_1 d\mu_{\log f_1}>-\infty$ if and only if  $\sum_{i=1}^{\infty}\sup \{\log f_1(x):x\in C_{i}\}\sup \{f_1(x):x\in C_{i}\}>-\infty$.
\end{enumerate}
\end{prop}

\begin{proof}
Let $g(x)=\log f_1(x)$ and $g_n(x)= e^{g(x)+g(\sigma x)+\dots+g(\sigma^{n-1}x)}$.
Define $\Phi=\{\log g_n\}_{n=1}^{\infty}$. We have that $\Phi$ is an additive Bowen sequence.
Indeed, we only need to prove that $\Phi$ is a Bowen sequence. In order to do so,  we note that for each $n\in\N$ we have
\begin{eqnarray*}
\sup \left\{\frac{g_n(x)}{g_n(y)}: x_i=y_i, 1\leq i \leq n \right\}
&=& \\ \sup \left\{\frac{f_1(x)f_1(\sigma x)\dots
f_1(\sigma^{n-1}x)}{f_1(y)f_1(\sigma x)\dots
f_1(\sigma^{n-1}y)}: x_i=y_i, 1\leq i\leq n \right\}.\end{eqnarray*}
 The claim now follows from assumption (\ref{pp1}). It is a direct consequence of assumption (\ref{pp2}) that
$-\infty<P(\Phi)<\infty$. Therefore, by Theorem \ref{main2}, there exists  a unique invariant Gibbs state $\mu_{\Phi}$ for $\Phi$. Since $P(\Phi)=P(\log f_1)$ we have that $\mu_{\Phi}$ is the unique
invariant Gibbs state for $\log f_1$. Moreover, since for any $\mu\in \M$ we have that $\lim_{n\rightarrow \infty}\frac{1}{n}\int \log g_n d\mu=\int \log f_1 d\mu$, then if $\int \log f_1 d\mu_{\Phi}>-\infty$, it follows that
$\mu_{\Phi}$ is the unique Gibbs equilibrium state for $\log f_1$.
This proves the first part of the Proposition. The second claim is proved in the exact same way as  the corresponding one in Proposition \ref{chara}.
%
\end{proof}

\begin{coro}\label{logf1}
Let  $(\Sigma, \sigma)$ be a topologically mixing countable Markov shift with the BIP property. Suppose that the
function
$\log f_1:\Sigma \to \R$ is of summable variation and satisfies  $\sum_{i\in \N} e^{\sup\log f_1\vert_{C_i}}<\infty$. Then,
\begin{equation*}
\sup_{n\in \N}\left\{\frac{f_1(x)f_1(\sigma x)\dots
f_1(\sigma^{n-1}x)}{f_1(y)f_1(\sigma x)\dots
f_1(\sigma^{n-1}y)}: x_i=y_i, 1\leq i\leq n \right\}<\infty
\end{equation*} and
 \begin{enumerate}
\item If $\int \log f_1 d\mu_{\log f_1}>-\infty$, then
$\mu_{\log f_1}$ is the unique Gibbs equilibrium state for $\log f_1$.
\item $\int \log f_1 d\mu_{\log f_1}>-\infty$ if and only if  $\sum_{i=1}^{\infty}\sup \{\log f_1(x):x\in C_{i}\}\sup \{f_1(x):x\in C_{i}\}>-\infty$.
\end{enumerate}

\end{coro}

\begin{proof}
Since $\log f_1$ is summable variation, there exists $N \in \R$ such that
$$\sum_{n=1}^{\infty}\sup\left\{\left\vert \log \frac{f_1(x)}{f_1(y)} \right\vert : x_i=y_i, 1\leq i\leq n \right\}\leq N.$$
Therefore, for $x, y\in \Sigma$ such that $x_i=y_i$, $1\leq i\leq n$ we have that
\begin{eqnarray*}
 \log \frac{f_1(x)f_1(\sigma x)\dots f_1(\sigma^{n-1}x)}{f_1(y)f_1(\sigma y)\dots f_1(\sigma^{n-1}y)} &\leq&\\
 \sup \left\{\log \frac{f_1(x)}{f_1(y)}:x_i=y_i, 1\leq i\leq n \right\} &+&\\
\sup \left\{\log \frac{f_1(\sigma x)}{f_1(\sigma y)}:(\sigma x)_i
=(\sigma y)_i, 1\leq i\leq n-1 \right\}+ &\dots& \\
+\dots+\sup\left\{\log \frac{f_1(\sigma^{n-1} x)}{f_1(\sigma^{n-1} y)}:(\sigma^{n-1} x)_i=(\sigma^{n-1} y)_i,
i=1 \right\}
&\leq N.
\end{eqnarray*}
Then
$$\sup_{n\in \N}\left\{\frac{f_1(x)f_1(\sigma x)\dots f_1(\sigma^{n-1}x)}{f_1(y)f_1(\sigma y)
\dots f_1(\sigma^{n-1}y)}: x_i=y_i, 1\leq i\leq n \right\} \leq e^{N},$$
and the result follows by Proposition \ref{general}.
\end{proof}

The next theorem characterises the existence of a Gibbs equilibrium state for $\F$ in terms of the existence of the Gibbs equilibrium
measure for $\log f_1$.

\begin{teo} \label{main}
Let $\Sigma$ be a topologically mixing countable Markov shift with the BIP property.
Let $\F=\{\log f_n\}_{n=1}^{\infty}$ be an almost-additive Bowen sequence on $\Sigma$ satisfying  $\sum_{i\in \N} \sup f_1\vert_{C_i}<
\infty$ and
\begin{equation*}
\sup_{n\in \N}\left\{\frac{f_1(x)f_1(\sigma x)\dots
f_1(\sigma^{n-1}x)}{f_1(y)f_1(\sigma x)\dots
f_1(\sigma^{n-1}y)}: x_i=y_i, 1\leq i\leq n \right\}<\infty.
\end{equation*}
Then $\F$ has a unique Gibbs equilibrium state $\mu_{\F}$ for $\F$ if and only if $\log f_1$ has a unique Gibbs equilibrium
state $\mu_{\log f_1}$ for $\log f_1$.
\end{teo}

\begin{proof}
To begin with, note that $P(\log f_1) <\infty$ and that the assumptions made on $f_1$ together with Proposition \ref{general} imply that there exists a
Gibbs state $\mu_{\log f_1}$ for $\log f_1$.

Let us first assume that there exists a unique  Gibbs equilibrium state, $\mu_{\F}$, for $\F$ and prove that $\mu_{\log f_1}$ is an equilibrium measure for $\log f_1$. Since  $\mu_{\F}$ is a Gibbs equilibrium state we have that  $\lim_{n\rightarrow \infty}\frac{1}{n}\int \log f_n d\mu_{\F}>-\infty$. Proposition \ref{chara} then implies that
 $$\sum_{i=1}^{\infty}\sup \{\log f_1(x):x\in C_{i}\}\sup \{f_1(x):x\in C_{i}\}>-\infty.$$
Thus, there exist constants $C_0', M' \in \R^+$ such that
\begin{equation*}
\int \log f_1 d\mu_{\log f_1}\geq \frac{e^{-P(\log f_1)}}{C_0'M'}\sum_{i=1}^{\infty}
\sup \left\{\log \frac{f_1(x)}{M'}:x\in C_i \right\} \sup\{f_1(x):x\in C_i\}.
\end{equation*}
Therefore, $\int \log f_1 d\mu_{\log f_1}>-\infty$ and we can conclude that $\mu_{\log f_1}$ is the unique Gibbs equilibrium state for $\log f_1$.

Conversely, suppose $\log f_1$ has a unique Gibbs equilibrium state $\mu_{\log f_1}$.
Then $\int \log f_1 d\mu_{\log f_1}>-\infty$ and Proposition \ref{general} implies that
 $$\sum_{i=1}^{\infty}\sup \{\log f_1(x):x\in C_{i}\}\sup \{f_1(x):x\in C_{i}\}>-\infty.$$
Therefore, applying Theorem \ref{main2} and Proposition \ref{chara} we obtain that
$\mu_{\F}$ is the unique Gibbs equilibrium state for $\F$.
\end{proof}

\begin{rem}
A similar result to that in Corollary \ref{logf1} but under different regularity assumptions was proved in \cite{mu}.
\end{rem}

\begin{coro}
Let $\Sigma$ be a topologically mixing countable Markov shift with the BIP property.
Let $\F=\{\log f_n\}_{n=1}^{\infty}$ be an almost-additive Bowen sequence on $\Sigma$, where $\log f_1$ is a function of summable variation that satisfies
$\sum_{i\in \N} \sup f_1\vert_{C_i}<
\infty$. Then $\F$ has a unique Gibbs equilibrium state $\mu_{\F}$ for $\F$ if and only if $\log f_1$ has a unique Gibbs equilibrium
state $\mu_{\log f_1}$ for $\log f_1$.
\end{coro}
\begin{proof}
The result immediately follows from Proposition \ref{general} and Theorem\ref{main}.
\end{proof}

\section{Zero temperature limits of Gibbs equilibrium states}\label{zero}
This section is devoted to state and prove our main result. We prove that for a certain class of almost-additive
potentials we can associate a family of Gibbs equilibrium states and that this sequence has, at least, one
accumulation point. It turns out that this measure is a maximising one. This result generalises the zero temperature limit theorems obtained for a single function in the compact (see \cite[Section 4]{j1} ) and in the non-compact settings (see \cite{jmu,big,bf,i,ke}). The major difficulty we have to face in this context is that the space of invariant probability measures is not compact, hence the existence of an accumulation point
is far from trivial. The techniques we use in the proof are  inspired in results by Jenkinson, Mauldin and
Urba\'nski \cite{jmu}. We begin defining the class of sequence of potentials that we will be interested in.

\begin{defi}\label{property}
Let $(\Sigma, \sigma)$ be a countable Markov shift satisfying the BIP property.
A sequence of continuous functions  $\F=\{\log f_n\}_{n=1}^{\infty}$, with $f_i:\Sigma \to \R$, $i\in\N$,  belongs to the class $\mathcal{R}$ if it satisfies the following properties:
\begin{enumerate}
\item The sequence $\F$ is almost-additive, Bowen and
$\sum_{i\in \mathbb{N}} \sup f_1\vert_{C_i}<\infty$. \label{a1}
\item  $\sum_{i=1}^{\infty}\sup \{\log f_1(x):x\in C_i\}\sup \{f_1(x):x\in C_i\}>-\infty$. \label{a2}
\end{enumerate}
\end{defi}

\begin{rem}
Note that it follows from Theorem \ref{main2} and Proposition \ref{chara} that if $\mathcal{F} \in \mathcal{R}$ then
there exists a unique invariant Gibbs equilibrium state $\mu_{\F}$ for $\F$.
Similarly, for every $t\geq 1$, there exists a unique invariant Gibbs equilibrium state $\mu_{t\F}$ for $t\F$.
\end{rem}
%


We begin  proving the upper semi-continuity of the limit of the integrals.
This is an essential result and it holds under weaker assumptions than those considered in the definition of the class $\mathcal{R}$.

\begin{lema}\label{lema1}
Let $(\Sigma, \sigma)$ be a countable Markov shift and $\F=\{\log f_n\}_{n=1}^{\infty}$ an almost-additive sequence of continuous
functions
on $\Sigma$ with $\sup f_1<\infty$. Then the map $m: \M \rightarrow \mathbb{R}$ defined by
$m(\mu)=\lim_{n\rightarrow \infty}
\frac{1}{n}\int \log f_n d\mu$ is upper semi continuous.
\end{lema}

\begin{proof}
We use an argument similar to that in the proof of  \cite[Proposition 3.7]{Y}. Let $\{\mu_{i}\}_{i\geq 1}$ be a sequence of measures in $\M$ which converge
to the measure  $\mu\in \M$ in the weak-star topology.

Let $M_1 \in \R$ be such that $\sup f_1\leq M_1$ and $C \in \R$ the almost-additive constant  that appear in equations (\ref{A1}) and (\ref{A2}) in Definition \ref{aaa}. Let $g_n(x)=f_n(x)e^C$ and note that $\{\log g_n\}_{n=1}^{\infty}$ is a sub-additive sequence. Moreover, $(\log g_1(x))^{+}\leq \log M_1 +\log e^C$ and so
$(\log g_1)^{+} \in L_{1}(\mu_i)$ for each $i\geq 1$.  Also, for each fixed $n\in \N$ we have that
$(\log g_n(x))^+\leq n \log M_1 +nC<\infty$
 which implies that for $i\geq 1$ we have $(\log g_n)^{+} \in L_{1}(\mu_i)$.
The sub-additive ergodic theorem (see \cite[Theorem 10.1]{w2}) implies that for each fixed
$\mu_i \in \M$ and $k \in \N$ we have,
\begin{equation*}
 \lim_{n\rightarrow \infty}\frac{1}{n}\int \log g_n d\mu_i\leq \frac{1}{k}\int \log g_k d\mu_i.
\end{equation*}
Thus,
\begin{equation*}
 \lim_{n\rightarrow \infty}\frac{1}{n}\int \log f_n d\mu_i\leq \frac{1}{k}\int \log f_k d\mu_i +\frac{C}{k}.
\end{equation*}
Therefore,
\begin{equation*}
\limsup_{i\rightarrow \infty}\lim_{n\rightarrow \infty} \frac{1}{n}\int \log f_n d{\mu}_i\leq \frac{1}{k}
\int \log f_k d\mu+\frac{C}{k}.
\end{equation*}
Letting $k \rightarrow \infty$, we obtain
$$\limsup_{i\rightarrow \infty}\lim_{n\rightarrow \infty}\frac{1}{n}\int \log f_n d\mu_i\leq
\lim_{k\rightarrow \infty}\frac{1}{k}
\int \log f_k d\mu,$$
which shows that $\limsup_{i\rightarrow \infty}m(\mu_i) \leq m(\mu)$, concluding the proof.
 \end{proof}

In our next lemma, we consider a one parameter family of Gibbs states $\mu_{t\F}$ corresponding to a family $t\F$.
We show how do the constants in the Gibbs property depend upon the parameter $t$. The proof of the lemma strongly uses the combinatorial assumptions we made on the system and an approximation argument we now describe.   Since $(\Sigma, \sigma)$ is topologically mixing and has the BIP property, there exist $k\in \N$ and a finite collection  $W$ of admissible words of length $k$ such that for any $a, b\in S$, there exists $w\in W$ such that $awb$ is admissible  (see \cite[p.1752]{s3} and \cite{mu}). Denote by  $A$ the transition matrix for $\Sigma$.  It is known (see \cite{iy,s1}) that  rearranging the set $\N$, there is an increasing sequence $\{l_{n}\}_{n=1}^{\infty}$ such that the matrix $A\vert_{ \{1, \dots, l_{n}\}\times \{1, \dots, l_{n}\}}$  is primitive (for a definition of primitive see \cite[p.5]{mu2}). Let $Y_{l_n}$ be the topologically
mixing finite state Markov shift with the transition matrix $A\vert_{ \{1, \dots, l_{n}\}\times \{1, \dots, l_{n}\}}$. Then there exists $p\in \N$ such that
for all $n\geq p$, the set  $Y_{l_n}$ contains all admissible words in  $W$.  We denote by $B_n(Y_l)$ the
set of admissible words of length $n$ in $Y_l$. Given $w\in W$, we let $N_{w}=\sup\{f_k(z):z\in C_w\}$ and
$\bar{N}=\min\{N_{w}: w\in W\}$. The proof of the next lemma makes use of ideas  from \cite [Claim 4.1]{iy}.

\begin{lema}\label{key}
Let $(\Sigma, \sigma)$ be a countable Markov shift with the BIP property and $\F \in \mathcal{R}$.
Let $t \geq 1$ and $\mu_{t\F}$ be the unique Gibbs equilibrium state for $t\F$. Then, for every $n \in \N$ and
for all $ x\in C_{i_1 \dots i_n}$, we have
\begin{equation*}
 \frac{\mu_{t\F}(C_{i_1 \dots i_n})}{e^{-nP(t\F)}f^{t}_n(x)}\leq \left(\frac{Me^{6C}}{D^5} \right)^t,
\end{equation*}
where
\begin{equation*}
D=\frac{\bar{N}e^{-3C}}{M^3 e^{(k-1)C}
\max\{\sum_{i\in \N}\sup f_1\vert_{C_i}, (\sum_{i\in \N}\sup f_1\vert_{C_i})^k\}}.
\end{equation*}
\end{lema}

\begin{proof}
Fix $t\in \N$ and $Y_{l_m}$ with $m\geq p$. In order to simplify the notation we denote $Y_{l_m}$ by $Y$.  Recall that $Y$ is a compact set.
Define $\alpha^{Y}_{n,t}:=\sum_{i_1\cdots i_{n}\in B_{n}(Y)}\sup\{f^{t}_n\vert _Y(y):
y\in C_{i_1\dots i_{n}}\}$.
For $l \in \N$, let $\nu_{l,t}$ be the Borel probability measure on $Y$ defined by
\begin{equation*}
\nu_{l,t}(C_{i_1\dots i_{l}})=\frac{\sup\{f^{t}_{l}\vert_{Y}(y):y\in C_{i_1\dots i_{l}}\}}{\alpha^{Y}_{l,t}}.
\end{equation*}
Let $a_{i_1\dots i_l}:=\sup\{f^{t}_{l}\vert_{Y}(y):y\in C_{i_1\dots i_{l}}\}$.
Let $l, n\in\N, l>n$. It is easy to see that
\begin{equation}\label{p4}
\alpha^{Y}_{l, t}\leq e^{Ct}\alpha^{Y}_{n, t}\alpha^Y_{l-n, t}.
\end{equation}
Now let $l,n\in \N, l>k$.
Using the same arguments used to prove equations (15), (16), (17) of \cite[Claim 4.1]{iy},
for each fixed $i_1\dots i_n\in B_n(Y)$,
we have
\begin{equation}\label{p1}
\sum_{t_1\dots t_l} \sup\{f^{t}_{n+l}\vert _Y(y):y\in C_{i_1\dots i_n t_1\dots t_l}\}\geq
\frac{\bar{N}^t e^{-2Ct}}{M^{3t}}a_{i_1\dots i_n} \alpha^Y_{l-k,t}.
\end{equation}
Thus  we obtain
\begin{equation} \label{p2}
\alpha^{Y}_{n+l, t}\geq \frac{\bar{N}^te^{-3Ct}\alpha^{Y}_{n, t}\alpha^{Y}_{l, t}}{M^{3t}\alpha^{Y}_{k,t}}.
\end{equation}
Also, from the proof of \cite[Claim 4.1]{iy}, we have that
\begin{equation}\label{p3}
\alpha^{Y}_{k,t} \leq e^{(k-1)Ct}\left(\sum_{i\in \N}\sup f_1\vert_{C_i}\right)^{tk}.
\end{equation}
 Using (\ref{p1}), (\ref{p2}) and (\ref{p3}), we obtain for $l,n \in\N,l>k$
\begin{equation}\label{p5}
\alpha^{Y}_{n+l, t}\geq D^{t}\alpha^{Y}_{n, t}\alpha^{Y}_{l,t}.
\end{equation}
We now claim that (\ref{p5}) holds for $n\in\N, 1\leq l\leq k$.
Let $1\leq r\leq k$. Then
\begin{equation*}
\alpha^{Y}_{n+r, t}\geq e^{-Ct}\sum_{i_1\dots i_{n+r}\in B_{n+r}(Y)}\sup\{f^{t}_n\vert_Y (y) f^t_{r}\vert _Y (\sigma^n y):
y\in C_{i_1\dots i_{n+r}}\}.
\end{equation*}
Let $i_1\dots i_n\in B_n(Y)$ and $u$ a fixed admissible word of $\Sigma$. Then we can connect $i_1\dots i_n$ and $u$ by
an admissible word $w=w_1\dots w_k\in W$ of length $k$ such that $i_1\dots i_nw_1\dots w_ku$ is allowable.
Now
\begin{equation*}
\sum_{i_1\dots i_{n+r}\in B_{n+r}(Y)}\sup\{f^{t}_n\vert_Y (y) f^t_{r}\vert _Y (\sigma^n y):
 y\in C_{i_1\dots i_{n+r}}\}\geq \sum_{\substack{\bar y\in C_{i_1\dots i_nw_1\dots w_r}\\ i_1\dots i_n\in B_{n}(Y)}}f^{t}_n\vert_Y (\bar y)
f^t_{r}\vert _Y (\sigma^n \bar y),
\end{equation*}
where $w_1\dots w_{r}\dots w_k\in W$ is chosen for each $i_1\dots i_n\in B_n(Y)$ as explained in the preceding paragraph
and $\bar y$ is any point from $C_{i_1\dots i_nw_1\dots w_r}$.
Therefore,
\begin{equation*}
\sum_{\substack{\bar y\in C_{i_1\dots i_nw_1\dots w_r}\\i_1\dots i_n\in B_n(Y)}}f^{t}_n\vert_Y (\bar y)
f^t_{r}\vert _Y (\sigma^n \bar y)
\geq (\frac{\bar N}{M})^t
\sum_{\substack{\bar y\in C_{i_1\dots i_n}\\i_1\dots i_n\in B_n(Y)}}f^{t}_n\vert_Y (\bar y)\geq
(\frac{\bar N}{M^2})^t \alpha^{Y}_{n,t},
\end{equation*}
where for both inequalities we use the fact that $\F$ has bounded variation.
Noting that
\begin{equation*}
\alpha^{Y}_{r,t}\leq e^{(r-1)Ct}(\sum_{i\in \N}\sup f_1\vert_{C_i})^{tr}\leq  e^{(k-1)Ct}
(\max\{\sum_{i\in \N}\sup f_1\vert_{C_i}, (\sum_{i\in \N}\sup f_1\vert_{C_i})^k\})^t ,
\end{equation*}
we obtain
\begin{equation*}
\frac{\alpha^{Y}_{n+r,t}}{\alpha^{Y}_{n,t}\alpha^{Y}_{r,t}}\geq (\frac{\bar N e^{-C}}{M^2e^{(k-1)C}
  \max\{\sum_{i\in \N}\sup f_1\vert_{C_i}, (\sum_{i\in \N}\sup f_1\vert_{C_i})^k\}})^t
\geq D^t
\end{equation*}
Now we proved the claim.
Using the arguments in the proof of \cite[Claim 4.1]{iy}, for all $n\in \N$, we obtain
\begin{equation} \label{p6}
D^t\alpha^Y_{n, t}\leq e^{nP(t\F\vert_Y)}\leq e^{Ct}\alpha^{Y}_{n,t}.
\end{equation}
Now, let $l> n+k$. For each fixed $i_1\dots i_n\in B_n(Y)$, we have
\begin{align*}
& \nu_{l,t}(C_{i_1\dots i_n})\\
&\leq \sum_{j_1\dots j_{l-n}} \frac{\sup\{f^t_{l}\vert _Y (y): y\in C_{i_1\dots i_nj_1\dots j_{l-n}}\}}{\alpha^{Y}_{l, t}}\\
&\leq \sum_{j_1\dots j_{l-n}} \frac{e^{Ct}\sup\{f^{t}_n\vert_Y (y) f^t_{l-n}\vert _Y (\sigma^n y):
 y\in C_{i_1\dots i_nj_1\dots
 j_{l-n}}\}}{\alpha^Y_{l,t}}\\
& \leq e^{Ct}\sum_{j_1\dots j_{l-n}} \frac{\sup\{f^{t}_n\vert_Y (y) :
 y\in C_{i_1\dots i_n}\} \sup\{f^{t}_{l-n}\vert_Y (\sigma^n y) :
 y\in C_{i_1\dots i_n j_1\dots j_{l-n}}\}}{\alpha^Y_{l,t}}\\
&\leq \frac{e^{Ct}}{\alpha^Y_{l,t}}\sup\{f^{t}_n\vert_Y (y) :
 y\in C_{i_1\dots i_n}\}\sum_{j_1\dots j_{l-n}} \sup\{f^{t}_{l-n}\vert_Y (y) :
 y\in C_{j_1\dots \dots j_{l-n}}\}\\
&\leq \frac{e^{Ct}}{D^t \alpha^{Y}_{n,t}}\sup\{f^t_n\vert_Y(y):y\in C_{i_1\dots i_n}\} \text{ (by } (\ref{p5}))\\
&\leq \frac{e^{2Ct}e^{-nP(t\F\vert_Y)}}{D^t}\sup\{f^t_n\vert_Y(y):y\in C_{i_1\dots i_n}\} \text{ (by } (\ref{p6})).
\end{align*}
Therefore, we obtain
\begin{equation}\label{p7}
\frac{\nu_{l,t}(C_{i_1\dots i_n})}{a_{i_1\dots i_n}e^{-nP(t\F\vert_Y)}}
\leq \frac{e^{2Ct}}{D^t}.
\end{equation}
Using the property of bounded variation, for all $y\in C_{i_1\dots i_n}$, we have
\begin{equation}\label{p11}
\frac{\nu_{l,t}(C_{i_1\dots i_n})}{f^{t}_n\vert_Y(y)e^{-nP(t\F\vert_Y)}}
\leq \frac{e^{2Ct}M^t}{D^t}.
\end{equation}
Also,
\begin{align*}
&\nu_{l,t}(C_{i_1\dots i_n})\geq \frac{\bar N^{t}e^{-2Ct}}
{\alpha^Y_{l, t}M^{3t}}a_{i_1\dots i_n}\alpha^{Y}_{l-n-k,t}
\text{ (by } (\ref{p1}))
\geq \frac{ \bar N^{t}e^{-2Ct}a_{i_1\dots i_n}\alpha^{Y}_{l-n,t}}
{\alpha^Y_{l,t} \alpha^Y_{k,t}M^{3t}e^{Ct}} \text { (by } (\ref{p4}))\\
&\geq \frac{D^t a_{i_1\dots i_n}}{e^{Ct}\alpha^{Y}_{n,t}} \text { (by } (\ref{p4}) \text{ and } (\ref{p3}))
\geq \frac{D^{2t}}{e^{Ct}}a_{i_1\dots i_n}e^{-nP(t\F\vert_Y)} \text{ (by } (\ref{p6})).
\end{align*}
Thus
\begin{equation}\label{p10}
\frac{\nu_{l,t}(C_{i_1\dots i_n})}{a_{i_1\dots i_n}e^{-nP(t\F\vert _Y)}}\geq \frac{D^{2t}}{e^{Ct}}.
\end{equation}
Hence we obtain for each $y\in C_{i_1\dots i_n}, l>n+k$,
\begin{equation}\label{p8}
\frac{D^{2t}}{e^{Ct}}\leq \frac{\nu_{l,t}(C_{i_1\dots i_n})}
{f^t_n\vert _Y(y)e^{-nP(t\F\vert_Y)}}\leq \frac{e^{2Ct}M^t}{D^t}.
\end{equation}
Consider now a convergent subsequence $\{\nu_{l_k,t}\}_{k=1}^{\infty}$ of
$\{\nu_{l,t}\}_{l=1}^{\infty}$ and let $\nu_t$ be
the corresponding limit point. Then $\nu_t$ also satisfies (\ref{p7}),(\ref{p11}), (\ref{p10}) and (\ref{p8}).
We know by \cite[Lemma 2]{b2} that $\nu_t$ is ergodic.
Using the arguments in \cite{b2}, we construct $\sigma$-invariant ergodic Gibbs measure $\mu_{t\F\vert_Y}$
 for $t\F\vert_Y$.
A limit point of the sequence $\{\frac{1}{n}\sum_{l=0}^{n-1}{\nu_t}\circ \sigma^{-l}\}_{n=1}^{\infty}$ is the unique
equilibrium state for  $t\F\vert_Y$ which is also Gibbs (see  \cite{iy}).
Let $i_1\dots i_n\in B_n(Y)$ be fixed. For $l>k$,
\begin{align*}
&\nu_t(\sigma^{-l}(C_{i_1\dots i_n}))\\
&\leq \frac{e^{2Ct}}{D^{t}}\sum_{j_1\dots j_{l} i_{1}\dots i_n}
\sup\{f^t_{l+n}\vert_Y(y):
y\in C_{j_1\dots j_l i_{i}\dots i_{n}}\}e^{-(l+n)P(t\F\vert _Y)} \text{ (replacing } \nu_{l,t} \text{ by }
\nu_t \text{ in } (\ref{p7}))\\
&\leq \frac{e^{3Ct}}{D^t} \sum_{j_1\dots j_l}\sup\{f^t_{l}(y)f^t_n(\sigma^{l}y):
y\in C_{j_1\dots j_{l}i_1\dots i_n}\}
e^{-(l+n)P(t\F\vert _Y)} \\
&\leq \frac{e^{3Ct}}{D^t}\sup\{f^t_n(y): y\in C_{i_1\dots i_n}\}\alpha^Y_{l,t} e^{-(l+n)P(t\F\vert _Y)}\\
&\leq \frac{e^{3Ct}}{D^{2t}}\sup\{f^t_n(y): y\in C_{i_1\dots i_n}\}e^{-nP(tF\vert_Y)} \text{ (by } (\ref{p6})).
\end{align*}
Therefore, using (\ref{p10}) (replacing $\nu_{l,t}$ by $\nu_t$), we have
\begin{equation*}
\nu_t(\sigma^{-l}(C_{i_1\dots i_n}))\leq \frac{e^{4Ct}}{D^{4t}}\nu_t(C_{i_1\dots i_n}).
\end{equation*}
Using (\ref{p8}) (replacing $\nu_{l,t}$ by $\nu_t$), for all $y\in C_{i_1\dots i_n}$,
\begin{equation*}
\frac{1}{m}
\sum_{l=0}^{m-1}\nu_t(\sigma^{-l}(C_{i_1\dots i_n}))
\leq
\frac{m-k}{m}(\frac{e^{6Ct}M^t f^t_n\vert_Y(y)
e^{-nP(t\F\vert_Y)}}{D^{5t}})+\frac{k}{m}
\end{equation*}
and hence for all $y\in C_{i_1\dots i_n}$,
\begin{equation}\label{p9}
\mu_{t\F\vert_Y}(C_{i_1\dots i_n})
\leq
\frac{e^{6Ct}M^t f^t_n\vert_Y(y)
e^{-nP(t\F\vert_Y)}}{D^{5t}}.
\end{equation}
Therefore, for each fixed $l_m, m\geq p$, $t\F\vert_{Y_{l_m}}$ has a unique equilibrium state
$\mu_{t\F\vert_{Y_{l_m}}}$ which is Gibbs
and  satisfies (\ref{p9}) (replacing $\mu_{t\F\vert_Y}$ by $\mu_{t\F\vert_{Y_{l_m}}}$).
The proof of \cite[Theorem 4.1]{iy} shows that (\ref{p9}) holds when we replace $\mu_{t\F\vert_Y}$ and $f^{t}_n\vert_Y(y)$
by the unique Gibbs equilibrium state
$\mu_{t\F}$ for $t\F$ and $f^{t}_n(y)$ respectively. This proves the lemma.
\end{proof}

\begin{lema}\label{L1}
Let $(\Sigma, \sigma)$ be a countable Markov shift with the BIP property and $\F \in \mathcal{R}$. Then
the family of Gibbs equilibrium states $\{\mu_{t\F}\}_{t\geq 1}$ is tight, i.e.,
for all $\epsilon>0$, there exists a compact set
$K\subset \Sigma$ such that for all $t\geq 1$ we have  $\mu_{t\F}(K)>1-\epsilon$ .
\end{lema}

\begin{proof}
The proof is based on  \cite[Lemma 2]{jmu}.
Let $\epsilon>0$. We construct an increasing sequence of positive integers $\{n_k\}_{k=1}^{\infty}$ such that the compact set
$$K=\{x\in \Sigma: 1\leq x_k\leq n_k, \text { for all } k\in \mathbb{N}\}$$
 satisfies $\mu_{t\F}(K)>1-\epsilon$ for all $t\geq 1$.
Let  $\pi_{k}: \Sigma \rightarrow \mathbb{N}$ be the projection  map onto the $k$-th coordinate.
Note that
\begin{align*}
\mu_{t\F}(K)=&\mu_{t\F} \left( \Sigma \cap (\cup_{k=1}^{\infty} \{x\in  \Sigma :x_k>n_k\})^{c}\right)\\
&\geq 1-\sum_{k=1}^{\infty}\mu_{t\F}(\{x\in \Sigma :x_k>n_k\})\\
&= 1-\sum_{k=1}^{\infty}\sum_{i=n_k+1}^{\infty}\mu_{t\F}(\pi_{k}^{-1}(i))\\
&= 1-\sum_{k=1}^{\infty}\sum_{i=n_k+1}^{\infty}\mu_{t\F}(C_i).
\end{align*}
Therefore, in order to show that $\{\mu_{t\F}\}_{t\geq1}$ is tight, it is enough to find $\{n_k\}_{k=1}^{\infty}$ such that
\begin{equation}\label{e1}
\sum_{i={n_k}+1}^{\infty}\mu_{t\F}(C_i)<\frac{\epsilon}{2^k}, \text{ for all } k\in \mathbb{N}, t\geq 1.
\end{equation}
Now, let $N={Me^{6C}}/{D^5}$ in Lemma \ref{key}.  If $n=1$,  we have
\begin{equation*}
 \mu_{t\F}(C_i)\leq N^te^{-P(t\F)}\sup\{f_1^{t}(x):x\in C_i\} \text{ for all } t\geq 1.
\end{equation*}
Now, let $m$ be any $\sigma$-invariant Borel probability measure for which the limit $$I=\lim_{n\rightarrow \infty}\frac{1}{n}\int \log f_n dm$$
is finite. Then
\begin{align*}
P(t\F)-tI=&\sup\left\{h_{\mu}(\sigma) + t\lim_{n\rightarrow \infty}\frac{1}{n}\int \log f_n d\mu : \mu \in \M\right\}-tI\\
&= P(t\F-tI)\geq h_m(\sigma)\geq 0.
\end{align*}
Therefore, for $t\geq 1$,
\begin{align}\label{e5}
\mu_{t\F}(C_i)&\leq N^te^{-P(t(\F-I))}e^{-tI} (\sup\{f_1(x):x\in C_i\})^t  \\
&\leq N^te^{-tI}(\sup\{f_1(x):x\in C_i\})^t\\
&=\left (Ne^{-I}\right)^t \left(\sup\{f_1(x):x\in C_i\}\right)^t
\end{align}
Note that Definition \ref{property} (\ref{a1}) implies that, given $\epsilon>0$, we can find $J \in \N$ such that
\begin{equation}\label{e6}
\sum_{i>J}\sup\{f_1(x):x\in C_i\}<\frac{\epsilon}{Ne^{-I}}\frac{1}{2^k}.
\end{equation}
Now we show equation (\ref{e1}).
Using (\ref{e5}) and (\ref{e6}), we obtain
\begin{align*}
\sum_{i>J}\mu_{t\F}(C_i)&\leq \left (Ne^{-I}\right)^t\sum_{i>J}(\sup\{f_1(x):x\in C_i\})^t\\
&=\left(\frac{\epsilon}{2^k}\right)^t\leq \frac{\epsilon}{2^k}.
\end{align*}
Thus we obtain (\ref{e1}).
\end{proof}

\begin{rem}
Lemma \ref{L1} implies that the family of Gibbs equilibrium states $\{\mu_{t\F}\}_{t\geq 1}$ has a subsequence that
converges weakly to a $\sigma$-invariant Borel probability measure $\mu$.
\end{rem}

We now state and prove our main result.

\begin{teo}\label{main4}
Let $(\Sigma, \sigma)$ be a countable Markov shift with the BIP property and let $\F \in \mathcal{R}$.
Denote by $\mu \in \M$  any accumulation point of the sequence of Gibbs equilibrium states $\{\mu_{t\F}\}_{t\geq 1}$. Then
\begin{equation} \label{the:eq}
\lim_{n\rightarrow \infty}\frac{1}{n}\int \log f_n d\mu=
\lim_{t\rightarrow \infty}\lim_{n\rightarrow \infty}\frac{1}{n}\int \log f_n d\mu_{t\F},
\end{equation}
and $\mu$ is a maximising measure for $\F$.
\end{teo}

\begin{proof}
We will divide the proof of the Theorem in several Lemmas and Remarks.

\begin{rem}
Recall that the pressure function $t \mapsto P(t\F)$ is convex (see \cite[Corollary 3.2]{iy}) and finite for $t \geq 1$, therefore it is differentiable for every $t > 1$  except, maybe, for a countable set.
\end{rem}

\begin{lema}[Derivative of the pressure] \label{lema:der_pres}
If the function $P(t \F)$ is differentiable at $t=t_0$, $t_0>1$,  then
\[ \frac{d P(t\F)}{dt} \Big|_{t=t_0} = \lim_{n \to \infty} \frac{1}{n} \int \log f_n d \mu_{t_0 \F}.\]
\end{lema}

\begin{proof}[Proof of Lemma \ref{lema:der_pres}]
The proof of this result is fairly standard, see for example \cite[Theorem 1.2]{Fe}. Let $\epsilon >0$. then
\begin{eqnarray*}
P(t_0\F)= h(\mu_{t_0 \F}) + t_0\lim_{n \to \infty} \frac{1}{n} \int \log f_n d \mu_{t_0 \F} \quad \text{ and } \\
 P((t_0+ \epsilon) \F) \geq  h(\mu_{t_0 \F}) + t_0 \lim_{n \to \infty} \frac{1}{n} \int \log f_n d \mu_{t_0 \F} + \epsilon  \lim_{n \to \infty} \frac{1}{n} \int \log f_n d \mu_{t_0 \F}.
\end{eqnarray*}
Thus,
\begin{equation*}
\frac{P((t_0+ \epsilon) \F) -P(t_0\F) }{\epsilon} \geq \lim_{n \to \infty} \frac{1}{n} \int \log f_n d \mu_{t_0 \F}.
\end{equation*}
Recalling that one-sided derivatives of the pressure do exist, we have
\begin{equation*}
P'_+(t_0 \F):=\lim_{\epsilon \to 0^+}\frac{P((t_0+ \epsilon) \F) -P(t_0\F) }{\epsilon} \geq \lim_{n \to \infty} \frac{1}{n} \int \log f_n d \mu_{t_0 \F}.
\end{equation*}
On the other hand, let $\epsilon <0$, then
\begin{equation*}
\frac{P((t_0+ \epsilon) \F) -P(t_0\F) }{\epsilon} \leq \lim_{n \to \infty} \frac{1}{n} \int \log f_n d \mu_{t_0 \F}.
\end{equation*}
Thus,
\begin{equation*}
P'_-(t_0 \F):=\lim_{\epsilon \to 0^-}\frac{P((t_0+ \epsilon) \F) -P(t_0\F) }{\epsilon} \leq \lim_{n \to \infty} \frac{1}{n} \int \log f_n d \mu_{t_0 \F}.
\end{equation*}
Hence
\begin{equation*}
P'_-(t_0 \F)\leq  \lim_{n \to \infty} \frac{1}{n} \int \log f_n d \mu_{t_0 \F} \leq P'_+(t_0 \F).
\end{equation*}
Since we are assuming that $P(t \F)$ is differentiable at $t=t_0$, $t_0>1$,  we obtain the desired result,
\begin{eqnarray*}
P'_+(t_0 \F)=P'_-(t_0 \F)= \lim_{n \to \infty} \frac{1}{n} \int \log f_n d \mu_{t_0 \F}.
\end{eqnarray*}
\end{proof}

\begin{rem}[Uniform upper bound on the derivative] \label{rem:ub}
If $\{t_k\}_{k=1}^{\infty}$ is an increasing sequence of positive real numbers such that $t_k \to \infty$ and
for which the pressure function $P(t \F)$ is differentiable at $t=t_k$, $t_k>1$, for every $k \in \N$, then
\[ \lim_{k \to \infty} \frac{d P(t\F)}{dt} \Big|_{t=t_k} \leq \sum_{n=1}^{\infty} \sup \log f_1 |_{C_n} < \infty.\]
Recall that the derivative of the pressure is an increasing function (being convex), hence
the left hand side limit exists. The bound above follows from the sub-additivity of the family $\F$ together with Definition \ref{property} (1).
\end{rem}

\begin{lema}[Asymptotic derivative] \label{rem:ad} The following limits exists and
\[\lim_{t \to \infty} \frac{P(t\F)}{t}=\lim_{t \to \infty} P'_+(t \F) =\lim_{t \to \infty} P'_-(t \F).\]
\end{lema}

\begin{proof}
Note that the functions $P'_+(t \F)$ and $P'_-(t \F)$ are increasing (since the pressure is convex) and bounded above (because of Remark \ref{rem:ub}). Therefore both limits exist. Moreover, if $t_1 <t_2$ we have that
$P'_+(t_1 \F) \leq P'_-(t_2 \F)$, thus both limits coincide. The fact that the limit coincide with that of the asymptotic derivative follows from the following convexity inequality
\begin{equation*}
P'_+(t_1 \F) \leq \frac{P(t_1 \F)  - P(t_2 \F) }{t_1 - t_2} \leq P'_-(t_2 \F).
\end{equation*}
\end{proof}

Recall that,
\begin{equation*}
\alpha(\F):= \sup \left\{ \lim_{n\rightarrow \infty}\frac{1}{n}\int \log f_n d\nu : \nu \in \M \right\}.
\end{equation*}

\begin{lema}[The asymptotic derivative bounds the optimal value] \label{rem:adbov}
For every $\mu \in \M$ we have that
\begin{equation} \label{eq:non}
\lim_{n \to \infty} \frac{1}{n} \int \log f_n d \mu \leq \lim_{t \to \infty} \frac{P(t\F)}{t}.
\end{equation}
\end{lema}
\begin{proof}
Indeed, assume by way of contradiction that there exists $\nu \in \M$ for which equation \eqref{eq:non} does not hold. 
Then, there exists $t_0 >1$ such that
\begin{equation*}
\lim_{n \to \infty} \frac{1}{n} \int \log f_n d \nu > \frac{P(t_0\F)}{t_0}.
\end{equation*}
Note that since $-\infty < \lim_{n \to \infty} \frac{1}{n} \int \log f_n d \nu < \infty$ and $P(t\F) <\infty$, we can conclude that $h(\nu) <\infty$. Thus,
 \begin{equation*}
 h(\nu) + t_0 \lim_{n \to \infty} \frac{1}{n} \int \log f_n d \nu  > P(t_0\F).
   \end{equation*}
This contradicts the variational principle. Therefore
\begin{equation*}
\alpha(\F) \leq   \lim_{t \to \infty} \frac{P(t\F)}{t} = \lim_{t \to \infty} P'_+(t \F)=\lim_{t \to \infty} P'_-(t \F).
\end{equation*}
\end{proof}

\begin{rem}[The entropy is decreasing] \label{rem:ed}
Note now  that for every $t >1$ we have  $h(\mu_{\F}) \geq h(\mu_{t\F})$. This is consequence of the convexity of the pressure (see \cite[Corollary 3.2]{iy}). Indeed, assume by way of contradiction that there exists $t_0 >1$ for which $h(\mu_{\F}) < h( \mu_{t_0\F})$. Consider the following straight lines:
\begin{eqnarray*}
l_1(t):= h(\mu_{\F}) + t \lim_{n\rightarrow \infty}\frac{1}{n}\int \log f_n d\mu_{\F} \text{ and } &\\
l_{t_0}(t):= h(\mu_{t_0\F}) + t \lim_{n\rightarrow \infty}\frac{1}{n}\int \log f_n d\mu_{t_0\F}.
\end{eqnarray*}
We first note that $l_1(t)$ and $l_{t_0}(t)$ are finite for $t\geq 0$ because $\F\in\mathcal{R}$.
Since $\mu_{\F}$ is the unique equilibrium measure for $ \F$ and $\mu_{t_0}$ is the unique equilibrium measure for $t_0 \F$, we have that
\begin{eqnarray*}
l_{t_0}(0) > l_1(0) \quad , \quad l_{t_0}(1) <  l_1(1) \quad  \text{ and } \quad l_{t_0}(t_0) > l_1(t_0).
\end{eqnarray*}
But that can not happen since two different straight lines can only intersect in one point and the above means that they intersect in at least two points.
\end{rem}

We finally prove the Theorem.

Let $\epsilon >0$, since the function  $P'_+(t \F)$ is non-decreasing and $\alpha(\F) \leq  \lim_{t \to \infty} P'_+(t \F)$ there exists $t_0 >1$ such that
\begin{equation*}
\alpha(\F) - \epsilon \leq \lim_{n \to \infty} \frac{1}{n} \int \log f_n d \mu_{t_0 \F}.
\end{equation*}
Therefore,
 \begin{eqnarray*}
\alpha(\F)-\epsilon \leq \frac{h(\mu_{t_0 \F})}{t} +  \lim_{n \to \infty} \frac{1}{n} \int \log f_n d \mu_{t_0 \F} \leq \frac{P(t \F)}{t} =&\\
\frac{h(\mu_{t\F})}{t} +  \lim_{n\rightarrow \infty}\frac{1}{n}\int \log f_n d\mu_{t\F} \leq \frac{h(\mu_{\F})}{t} +  \lim_{n\rightarrow \infty}\frac{1}{n}\int \log f_n d\mu_{t\F} \leq &\\
 \frac{h(\mu_{\F})}{t} +\alpha(\F).
 \end{eqnarray*}
Thus, for any $\epsilon >0$ we have that
\begin{equation*}
\alpha(\F)-\epsilon \leq  \limsup_{t \to \infty}  \frac{P(t \F)}{t} \leq \limsup_{t \to \infty} \lim_{n\rightarrow \infty}\frac{1}{n}\int \log f_n d\mu_{t\F} \leq \alpha(\F).
\end{equation*}
Similary,  we have that
\begin{equation*}
\alpha(\F)-\epsilon \leq  \liminf_{t \to \infty}  \frac{P(t \F)}{t}
\leq \liminf_{t \to \infty} \lim_{n\rightarrow \infty}\frac{1}{n}\int \log f_n d\mu_{t\F} \leq \alpha(\F).
\end{equation*}
Therefore, letting $\epsilon\rightarrow 0$,
\begin{equation*}
\liminf_{t \to \infty} \lim_{n\rightarrow \infty}\frac{1}{n}\int \log f_n d\mu_{t\F}
=\limsup_{t \to \infty} \lim_{n\rightarrow \infty}\frac{1}{n}\int \log f_n d\mu_{t\F}.
\end{equation*}
It follows from Lemma \ref{lema1} that,
\begin{equation*}
\lim_{t \to \infty} \lim_{n\rightarrow \infty}\frac{1}{n}\int \log f_n d\mu_{t\F} =\lim_{n\rightarrow \infty}\frac{1}{n}\int
\log f_n d\mu=\alpha(\F).
\end{equation*}
Hence we obtain (\ref{the:eq}) and  $\mu$ is an $\F-$maximising measure.

\end{proof}

We stress that the result obtained in Theorem \ref{main4} is new even in the compact setting,
where convergence of Gibbs states directly follows as a consequence of the fact  that the space of invariant measures is compact. Also note that in this compact setting, since the system has finite topological entropy,  every Gibbs measure is an equilibrium measure.

\begin{coro} \label{coro:compact}
Let $(\Sigma, \sigma)$ be a transitive sub-shift of finite type defined  over a
finite alphabet and $\F$ be an almost-additive Bowen sequence on $\Sigma$.
Denote by $\mu \in \M$  any accumulation point of the sequence of Gibbs equilibrium states $\{\mu_{t\F}\}_{t\geq 1}$. Then
\begin{equation*}
\lim_{n\rightarrow \infty}\frac{1}{n}\int \log f_n d\mu=
\lim_{t\rightarrow \infty}\lim_{n\rightarrow \infty}\frac{1}{n}\int \log f_n d\mu_{t\F},
\end{equation*}
and $\mu$ is a maximising measure for $\F$.
\end{coro}
\begin{proof}
Since $(\Sigma, \sigma)$ is a sub-shift of finite type over a finite alphabet, it is clear that
$\F$ satisfies (\ref{a1}) and (\ref{a2}) of Definition \ref{property}. Now we apply Theorem \ref{main4}.
   \end{proof}

\section{The Joint spectral radius} \label{sec:joint}
In this section we will show that the techniques developed in this paper have interesting applications in functional analysis. Even in the compact (finite alphabet) setting, Theorem \ref{main4} can be used to obtain results in spectral theory. We begin recalling some basic definitions. Let $A$ be a  $d \times d$ real matrix, the \emph{spectral radius} of $A$, is defined by
\[ \rho(A)= \max \{ |\lambda|: \lambda \text{ is an eigenvalue of } A \}.\]
It is well known that if  $\| \cdot \|$ is any sub-multiplicative matrix norm the following relation (sometimes called Gelfand property) holds
\[ \rho(A)= \lim_{n \to \infty} \| A \|^{1/n}.\]
Let $\A:=\{A_1, A_2, \dots ,A_k\}$ be a a set of $d \times d$ real matrices and $\| \cdot \|$ a sub-multiplicative  matrix norm. The \emph{joint spectral radius} $\varrho(\A)$ is defined by
\begin{equation*}
\varrho(\A):=\lim_{n \to \infty} \max \left\{ \|A_{i_n} \cdots  A_{i_1}  \|^{1/n} : i_j \in \{1, 2, \dots, k\} \right\}.
\end{equation*}
This notion, that generalises the notion of spectral radius to a set of matrices,  was introduced by G.-C. Rota and W.G. Strang in 1960 \cite{rs}.  Interest on it was  strongly renewed by its applications in the study of wavelets discovered by Daubechies and Lagarias \cite{dl1,dl2}. The value of $\varrho(\A)$ is independent of the choice of the norm  since all of them are equivalent. There exists a wide range of  applications of the joint spectral radius in different topics including not only wavelets \cite{p}, but for example,  combinatorics \cite{d}.  Lagarias and Wang \cite{lw} conjectured
that for any finite set of matrices $\A$ there exists integers $\{i_1, \dots, i_n\}$ such that
the periodic product $A_{i_1} \cdots A_{i_n}$ satisfies
\begin{equation*}
\varrho(\A)= \rho (A_{i_n} \cdots A_{i_1})^{1/n}.
\end{equation*}
This conjecture was proven to be false by Bousch and Mariesse \cite{bm} and explicit counterexamples were first obtained by  Hare, Morris, Sidorov and Theys \cite{hm}.

It is possible to restate the definition of joint spectral radius in terms of dynamical systems. Indeed, let $(\Sigma_k, \sigma)$ be the full-shift on $k$ symbols and consider the family of maps, $\phi_n:\Sigma_k \to \R$ defined by
\[\phi_n(x):= \|A_{i_n} \cdots  A_{i_1}  \|. \]
The family $\F:= \left\{ \log \phi_n \right\}_{n=1}^{\infty}$ is sub-additive on $\Sigma_k$.
Denote by $\M$ the set of $\sigma-$invariant probability measures. If $\nu \in \M$ is ergodic then Kingman's sub-additive theorem implies that $\nu$-almost everywhere the following equality holds
\begin{equation*}
\lim_{n \to \infty} \frac{1}{n} \int \log \phi_n \ d \nu = \lim_{n \to \infty} \frac{1}{n} \log \phi_n(x).
\end{equation*}

In this compact setting there exists an invariant measure realising the joint spectral radius. It seems that this was first observed by Schreiber \cite{sc} and later rediscovered by Sturman and Stark \cite{ss}. The most general version of this result has been obtained by Morris \cite[Appendix A]{mo}. This later result has the advantage that it allows the potentials $\phi_n$ to have a range in the interval $[-\infty,+\infty)$.

\begin{lema}[Schreiber's Theorem] \label{medida-max}
Let $\A:=\{A_1, A_2, \dots ,A_k\}$ be a a set of $d \times d$ real matrices then there exits a measure $\mu \in \M$ such that
\begin{equation*}
\varrho(\A)= \exp\left( \sup \left\{	\lim_{n \to \infty} \frac{1}{n} \int \log \phi_n \ d \nu : \nu \in \M\right\}\right) =  \exp \left( \lim_{n \to \infty}  \frac{1}{n}  \int \log \phi_n d \mu \right).
\end{equation*}
\end{lema}
We refer to  \cite[Appendix A]{mo} for a proof of Lemma \ref{medida-max}. It is worth stressing that the definition of joint spectral radius and Lemma \ref{medida-max} hold in a broader context. Indeed, we can consider  $(\Sigma, \sigma)$ to be any mixing sub-shift of finite type defined on a finite alphabet and define the corresponding joint spectral radius by
 \begin{equation*}
\varrho_{\Sigma}(\A):=\exp \left( \lim_{n \to \infty}   \max \left\{ \|A_{i_n} \cdots  A_{i_1}  \|^{1/n} : (i_1 i_2 \dots i_n) \text{ is an admissible word }  \right\} \right).
\end{equation*}

Under certain cone conditions for the set $\A$, we will show that there exists a  one parameter
family of dynamically relevant (Gibbs states) invariant measures $\{\mu_t\}_{t\geq 1}$ such that any weak star accumulation
point of it $\mu$, satisfies
\begin{equation}
\varrho(\A)=\exp \left( \lim_{n \to \infty} \frac{1}{n} \int \log \phi_n \ d \mu \right).
\end{equation}

\begin{teo}[Compact case]
Let $\A:=\{A_1, A_2, \dots ,A_k\}$ be a a set of $d \times d$ and $(\Sigma, \sigma)$ mixing sub-shift of finite
type defined on the alphabet $\{1,2 \dots, k\}$. Let $\phi_n :\Sigma_k \to \R$ be defined by
 \[\phi_n(w)= \phi_n((i_1, i_2, \dots))= \|A_{i_n} \cdots A_{i_1} \|.  \]
 If the family   $\mathcal{F}= \{\log \phi_n  \}_{n=1}^{\infty}$ is almost-additive then  for every $t \geq 1$
there exists a unique Gibbs state $\mu_t$ corresponding to $t\F$ and there exists a weak star accumulation
point $\mu$ for $\{ \mu_t \}_{t\geq 1}$. The measure $\mu$ is such that
\begin{equation*}
    \varrho(\mathcal{A})= \exp\left( \lim_{n \to \infty} \frac{1}{n} \int
\log \phi_ n d \mu \right)= \exp \left( \lim_{t  \to \infty} \left( \lim_{n \to \infty} \frac{1}{n} \int
\log \phi_ n d \mu_t \right) \right).
\end{equation*}
\end{teo}

\begin{proof}
Since the family $\F$ is almost-additive and is a Bowen sequence then it is a consequence of
Theorem \ref{main2} that there exists a unique Gibbs state $\mu_t$ corresponding to $t \F$ for every
$t \geq 1$ which is also an equilibrium measure. Since the space $\M$ is compact there exists a weak star accumulation
point $\mu$ for the sequence $\{\mu_t\}_{t\geq 1}$. The result now follows from Theorem \ref{main4} or Corollary \ref{coro:compact}.
\end{proof}

\begin{coro} \label{con}
Let $\mathcal{B}=\left\{A_1, A_2 , \dots, A_n  \right\}$ be a finite set of
 positive matrices then the family $\mathcal{F}= \{\log \phi_n  \}_{n=1}^{\infty}$
is almost-additive and therefore the sequence of
Gibbs measures $\{\mu_t\}_{t\geq 1}$ for $t \F$ has an accumulation point $\mu   \in \M$ and
\begin{equation*}
 \varrho(\mathcal{A})= \exp\left( \lim_{n \to \infty} \frac{1}{n} \int \log \phi_ n d \mu \right)= \exp \left( \lim_{t  \to \infty} \left( \lim_{n \to \infty} \frac{1}{n} \int \log \phi_ n d \mu_t \right) \right).
\end{equation*}
\end{coro}

\begin{proof}
It was proved in \cite[Lemma 2.1]{Fe1} that the set $\mathcal{B}$ is almost-additive. The result follows from Theorem \ref{teo:joint}.
\end{proof}

\begin{rem} \label{rem:pablo}
It is well known (see, for example,  \cite[Section 2]{fs}  or \cite[Section 3.3]{pablo} for precise statements and proofs) that if a finite set of matrices satisfies a \emph{cone condition} then the associated family of potentials $\mathcal{F}= \{\log \phi_n  \}_{n=1}^{\infty}$ is almost-additive. The particular case of positive matrices is in itself interesting, but the technical property  required to develop the theory is that it satisfies the so called cone condition. Moreover, it was recently shown by Feng and Shmerkin \cite[Section 3]{fs} that the sub-additive pressure for matrix cocycles is continuous on the set of matrices. The main technical idea of the proof is the construction of a subsystem (after iteration) which is almost-additive and that approximates the pressure. This indicates that the almost-additive theory of thermodynamic formalism can be used to approximate the corresponding sub-additive one.
\end{rem}

%


Let us consider now the non-compact case. Let  $\A:=\{A_1, A_2, \dots \}$ be a countable  set of $d \times d$ real matrices. We can again consider the joint spectral radius of them. However, the conclusion of Lemma \ref{medida-max} might not be true. Even for the case of one potential $\psi:\Sigma \to \R$, there are examples of non-compact dynamical systems for which
\[ \sup\left\{ \int \psi  d \mu : \mu \in \M  \right\} < \limsup_{n \to \infty} \frac{1}{n} \sup \left\{\sum_{i=0}^{n-1} \psi(\sigma^ix) : x \in \Sigma \right\}. \]
See for instance \cite[Example 4]{jmu2}. So we consider a slightly different situation.

\begin{teo} \label{teo:joint}
Let $\mathcal{A}=\left\{A_1, A_2 , \dots  \right\}$ be a countable set of matrices and
$(\Sigma, \sigma)$ a topologically mixing countable Markov shift satisfying the BIP property.  Let
\[\phi_n(w)= \|A_{i_n} \cdots A_{i_1} \|.  \]
If the family $\mathcal{F}= \{\log \phi_n  \}_{n=1}^{\infty}$ is almost-additive then for every $t \geq 1$
there exists a unique Gibbs measure $\mu_t$ corresponding to $t\F$ and
there exists a weak star accumulation point $\mu$ for $\{ \mu_t \}_{t\geq 1}$. The measure $\mu$ is such that
\begin{eqnarray*}
\sup \left\{ \lim_{n \to \infty} \frac{1}{n} \int \log \phi_ n d \nu : \nu \in \M   \right\} = \lim_{n \to \infty} \frac{1}{n} \int \log \phi_ n d \mu &=&\\ \lim_{t  \to \infty}  \lim_{n \to \infty} \frac{1}{n} \int \log \phi_ n d \mu_t.
\end{eqnarray*}
\end{teo}
The proof of this results follows from the zero temperature limit theorems obtained in the previous sections.

\begin{proof}
Since the family $\F$ is almost-additive and is a Bowen sequence then it is a consequence of Theorem \ref{main2} that
there exists a unique Gibbs state $\mu_t$ corresponding to $t \F$ for every $t \geq 1$ such
that $P(t \F) < \infty$. Lemma \ref{L1} implies that there exists a weak star accumulation point $\mu$ for the
sequence $\{\mu_t\}_{t\geq 1}$. The result now follows from Theorem \ref{main4}.
\end{proof}

\begin{coro}
Let $\mathcal{A}=\left\{A_1, A_2 , \dots  \right\}$ be a sequence of
matrices having strictly positive entries and such that there exists a constant $C >0$ with the property that for every $ k \in \N$ the following holds
\begin{equation*}
 \frac{\min_{i,j}(A_k)_{i,j}}{\max_{i,j}(A_k)_{i,j}} \geq C
\end{equation*}
then for every sufficiently large $t\in \R$ there exists a Gibbs state $\mu_t$ for $t \F$ and the sequence
$\{\mu_t\}_{t\geq 1}$ has an accumulation point $\mu   \in \M$. Moreover,
\begin{eqnarray*}
\sup \left\{ \lim_{n \to \infty} \frac{1}{n} \int \log \phi_ n d \nu : \nu \in \M   \right\} = \lim_{n \to \infty} \frac{1}{n} \int \log \phi_ n d \mu &=& \\ \lim_{t  \to \infty}  \lim_{n \to \infty} \frac{1}{n} \int \log \phi_ n d \mu_t.
\end{eqnarray*}
\end{coro}

\begin{proof}
Under the assumptions of the theorem the family $\mathcal{F}= \{\log \phi_n  \}_{n=1}^{\infty}$ is almost-additive (see \cite[Lemma 7.1]{iy}) on $\Sigma$ and therefore the results directly follows from Theorem \ref{teo:joint}.
\end{proof}

Let us stress that the space of invariant measures is not compact, so the existence of such an invariant measure is non trivial.

\section{Maximising the singular value function} \label{sec:singular}
Ever since the pioneering work of Bowen \cite{bow2} the relation between thermodynamic formalism and
the dimension theory of dynamical systems has been thoroughly studied and exploited (see for example \cite{ba,b3,pe}). Multifractal analysis is a sub-area of dimension theory where the results obtained out of this relation has been particularly successful. The main goal in  multifractal analysis is   to study the complexity of level sets of invariant local quantities. Examples of these quantities are Birkhoff averages, Lyapunov exponents, local entropies and pointwise dimension. In general the structure of these level sets is very complicated so tools such as Hausdorff dimension or topological entropy are used to describe them. In dimension two (or higher) where a typical dynamical system is non-conformal computing the exact value of   the Hausdorff dimension of the level sets is an extremely complicated task and at this point there are no techniques available to deal with such problem.

In this section we show how the results obtained in Section \ref{zero} can be used in the study of multifractal analysis of  Lyapunov exponents for certain non-conformal repellers. Indeed, we will combine our results on ergodic optimisation with those of Barreira and Gelfert \cite{bg} (which in turn uses ideas of Falconer \cite{f}) to construct a measure supported on the extreme level sets.

Let $f : \R^2 \mapsto  \R^2$ be a $C^1$ map and let  $\Lambda \subset \R^2$ be a repeller of $f$.  That is, the set $\Lambda$ is compact, $f$-invariant, and the map $f$ is expanding on $\Lambda$, i.e., there exist $c > 0$ and $\beta > 1$ such that
\[ \|d_xf^n(v) \| \geq c\beta^n \|v\|, \]
for every $x \in \Lambda$, $n \in \N$ and $v \in T_x \R^2$. We will also assume  that there exists  an open set $U \subset \R^2$  such that $\Lambda \subset U$ and $\Lambda = \cap_{n \in \N} f^n(U)$ and that $f$ restricted to $\Lambda$ is topologically mixing. A pair $(\Lambda, f)$ satisfying the above assumptions will be called \emph{expanding repeller}. All the above assumptions are standard and there is a large literature describing the dynamics of expanding maps (see for example \cite{bdv}). It is important to stress that the system $(\Lambda,f)$ can be coded with a finite state transitive Markov shift. For each $x \in  \R^2$ and $ v \in  T_xR^2$ we define the \emph{Lyapunov exponent} of $(x,v)$ by
\begin{equation*}
\lambda(x,v):=\limsup_{n \to \infty} \frac{1}{n} \log \| d_xf^n v  \|.
\end{equation*}
For each $x \in \R^2$ there exists a positive integer $s(x)  \leq 2$, numbers $\lambda_1(x) \geq \lambda_2(x)$, and linear subspaces
\[ \{0\} =E_{s(x)+1}(x) \subset  E_{s(x)}(x) \subset E_{1}(x)=T_xR^2,\]
such that
\[E_i(x)=\left\{v \in T_x\R^2 : \lambda(x,v)=\lambda_i(x) \right\}\]
and $ \lambda(x,v)=\lambda_i(x)$ if $v \in E_i(x) \setminus E_{i+1}(x)$.

In order to study study multifractal analysis of Lyapunov exponents in this context, Barreira and Gelfert \cite{bg} used a construction originally made by Falconer \cite{f} that we know recall. The singular values $s_1(A), s_2(A)$ of a $2\times 2$ matrix $A$ are the eigenvalues, counted with multiplicities, of the matrix $(A^*A)^{1/2}$, where $A^*$ denotes the transpose of $A$. The singular values can be interpreted as the length of the semi-axes of the ellipse which  is the image of the unit ball under $A$.  The functions, $\phi_{i,n}: \Lambda \to \R$ be defined by
 \[ \phi_{i,n}(x)= \log s_i(d_xf^n)  \]
and called \emph{singular value functions}. Falconer \cite{f} studied them with the purpose of estimating the Hausdorff dimension of $\Lambda$ and have become one of the major tools in the dimension theory for non-conformal systems.  It directly follows from Oseledets' multiplicative ergodic theorem \cite[Chapter 3]{bp} that for each finite $f-$invariant measure $\mu$ there exists a set $X \subset \R^2$ of full $\mu$ measure such that
\begin{equation} \label{eq:lya}
 \lim_{n \to \infty} \frac{\phi_{i,n}(x)}{n}= \lim_{n \to \infty} \frac{1}{n} \log s_i(d_xf^n)= \lambda_i(x).
 \end{equation}
Given $\alpha=(\alpha_1, \alpha_2) \in \R^2$ define the corresponding level set by
\[L(\alpha):=  \left\{ x \in \Lambda : \lambda_1(x)= \alpha_1 \text{ and } \lambda_2(x)= \alpha_2  \right\}. \]
Barreira and Gelfert \cite{bg} described the entropy spectrum of the Lyapunov exponents of $f$, that is they studied the function  $\alpha \to h_{top}(f | L(\alpha))$, where $h_{top}$ denotes the topological entropy of the set $L(\alpha)$. Their study exploited the relation established in equation \eqref{eq:lya}, where it is shown that the level sets for the Lyapunov exponent correspond to level sets of the ergodic averages of the sub-additive sequences  defined by  $S_1=\{\phi_{1,n}\}_{n=1}^{\infty}$ and  $S_2= \{\phi_{2,n}\}_{n=1}^{\infty}$.

The following result, which is a direct consequence of the theorems obtained in Section \ref{zero}, allows us to describe the maximal  Lyapunov exponent of the map $f$.

\begin{prop}
Let $(\Lambda, f)$ be an expanding repeller such that the singular value
function are almost-additive then for every $t \geq 1$ there exists a unique Gibbs measure $\mu_t$
 corresponding to $tS_1$. Moreover, the sequence $\{\mu_t\}_{t\geq 1}$ has an accumulation point $\mu$ and
\begin{equation*}
\sup \left\{ \lim_{n \to \infty } \frac{1}{n} \phi_{1,n}(x)  \right\} = \lim_{n \to \infty} \frac{1}{n} \int \phi_{1,n}(x) d \mu.
\end{equation*}

\end{prop}

In particular, we obtain invariant measure supported on the set of points for which the Lyapunov exponent is maximised.

Conditions on the dynamical system $f$ so that the sequences  $S_1$ and $S_2$ are almost-additive can be found, for example, in \cite[Proposition 4]{bg} where it is proved that

\begin{lema}
Let $(\Lambda, f)$ be an expanding repeller. If
\begin{enumerate}
\item for every $x \in \Lambda$ the derivative $d_xf$is represented
by a positive $2 \times 2$ matrix,  or
\item  if $\Lambda$ posses a dominated splitting (see \cite[p.234]{b3} or \cite{bdv} for a precise definition).
\end{enumerate}
Then the sequences $S_1$ and  $S_2$ are almost-additive.
\end{lema}
Actually, a cone condition of the type discussed in Corollary \ref{con} is enough to obtain almost-additivity. This is discussed also in \cite{bg}. Again, as in Remark \ref{rem:pablo}, it was shown by Feng and Shmerkin \cite{fs} that the singular value function can be approximated by almost-additive ones.

We have considered maps defined in $\R^2$, similar results can be obtained in any finite dimension. Let us stress that we have only used a compact version of  the results of Section \ref{zero}, namely  Corollary \ref{coro:compact}, which  hold true in the countable (non-compact) setting.


\begin{thebibliography}{11}

\bibitem[BLL]{bll} A. Baraviera, R. Leplaideur and A. Lopes \emph{Ergodic Optimization, Zero Temperature Limits and the Max-Plus Algebra.}  XXIX Coloquio Brasileiro de Matem\'atica (2013).


\bibitem[B1]{b2} L.\ Barreira, \emph{Nonadditive thermodynamic formalism: equilibrium and Gibbs measures.}  Discrete Contin. Dyn. Syst. 16 (2006), 279--305.


\bibitem[B2]{ba} L. Barreira  \emph{Dimension and recurrence in hyperbolic dynamics.} Progress in Mathematics, 272. Birkhauser Verlag, Basel, 2008.

\bibitem[B3]{b3}  L.\ Barreira,   \emph{Thermodynamic Formalism and Applications to Dimension Theory.} Progress in Mathematics 294, BirkhŠuser, 2011.


\bibitem[BD]{bd}  L.\ Barreira and P.\ Doutor \emph{Almost additive multifractal analysis.} J. Math. Pures Appl. (9) 92 (2009), no. 1, 1--17.

\bibitem[BG]{bg} L.\ Barreira and  K.\ Gelfert, \emph{Multifractal analysis for Lyapunov exponents on nonconformal repellers.}
Comm. Math. Phys. 267 (2006), 393--418.

\bibitem[BP]{bp}    L.\ Barreira and Y.\ Pesin, \emph{Nonuniform hyperbolicity. Dynamics of systems with nonzero Lyapunov exponents.} Encyclopedia of Mathematics and its Applications, 115. Cambridge University Press, Cambridge, 2007. xiv+513 pp.

\bibitem[BGa]{big} R. Bissacot and E. Garibaldi, \emph{Weak KAM methods and ergodic optimal problems for countable Markov shifts.} Bull. Braz. Math. Soc. (N.S.) 41 (2010), no. 3, 321--338

\bibitem[BF]{bf} R. Bissacot and R. Freire, \emph{On the existence of maximizing measures for irreducible countable Markov shifts: a dynamical proof.} `FirstView' article at Ergodic Theory and Dynamical Systems.

%


\bibitem[BDV]{bdv} Ch. Bonatti, L.J. Diaz, and M. Viana, \emph{Dynamics beyond uniform hyperbolicity. A global geometric and probabilistic perspective.} Encyclopaedia of Mathematical Sciences, 102. Mathematical Physics, III. Springer-Verlag, Berlin, 2005.



\bibitem[Bo]{bo}   T. Bousch  \emph{Le poisson n'a pas d'aretes.} Ann. Inst. H. Poincare Probab. Statist. 36 (2000), no. 4, 489--508.

\bibitem[BoM]{bm} T. Bousch and J. Mairesse, \emph{Asymptotic height optimization for topical IFS, Tetris heaps, and the
finiteness conjecture.}  J. Amer. Math. Soc., 15 (2002),  77--111


\bibitem[Bow1]{bow} R. Bowen \emph{Equilibrium states and the ergodic theory of Anosov diffeomorphisms.}  Second revised edition. With a preface by David Ruelle. Edited by Jean-RenŽ Chazottes. Lecture Notes in Mathematics, 470. Springer-Verlag, Berlin, 2008. viii+75 pp

\bibitem[Bow2]{bow2} R. Bowen \emph{Hausdorff dimension of quasicircles.} Inst. Hautes ƒtudes Sci. Publ. Math. No. 50 (1979), 11--25.

\bibitem[Br]{br} J. Bremont, \emph{Gibbs measures at temperature zero.} Nonlinearity 16 (2003), no. 2, 419--426.


\bibitem[ChH]{hc}  J.R. Chazottes and M. Hochman, \emph{On the zero-temperature limit of Gibbs states.} Comm. Math. Phys. 297 (2010), no. 1, 265--281.

\bibitem[ChGU]{chgu}    J.R. Chazottes, J.M.  Gambaudo and E. Ugalde,  \emph{Zero-temperature limit of one-dimensional Gibbs states via renormalization: the case of locally constant potentials. } Ergodic Theory Dynam. Systems 31 (2011), no. 4, 1109--1161.



\bibitem[DL1]{dl1} I. Daubechies and J. Lagarias, \emph{Sets of matrices all
infinite products of which converge.} Linear Algebra Appl. 161 (1992), pp. 227--263.

\bibitem[DL2]{dl2}  I. Daubechies and J. Lagarias, \emph{Corrigendum/addendum to:
Sets of matrices all infinite products of which converge.}  Linear Algebra Appl. 327 (2001), pp. 69--83.


\bibitem[DST]{d} J. M. Dumont, N. Sidorov, and A. Thomas, \emph{Number of representations related to a linear recurrent basis.} Acta Arith., 88 (1999), pp. 371--396.

\bibitem[EFS]{efs}   A. van Enter, R. Fern\'andez and A. Sokal, \emph{Regularity properties and pathologies of position-space renormalization-group transformations: scope and limitations of Gibbsian theory.}  J. Statist. Phys. 72 (1993), no. 5-6, 879--1167.

\bibitem[F]{f} K. J. Falconer, \emph{A subadditive thermodynamic formalism for mixing repellers.} J. Phys. A 21
(1988), no. 14, 737--742

\bibitem[Fe1]{Fe} D.J. Feng, \emph{The variational principle for products of non-negative matrices.}  Nonlinearity  17 (2004) 447--457.

\bibitem[Fe2]{Fe1} D.J. Feng, \emph{Lyapunov exponents for products of matrices and multifractal analysis. I. Positive matrices.} Israel J. Math. 138 (2003), 353--376.

\bibitem[FeSch]{fs} D.J. Feng and P. Shmerkin \emph{Non-conformal repellers and the continuity of pressure for matrix cocycles.} Preprint.


\bibitem[Gr]{gr} G. Gripenberg, \emph{Computing the joint spectral radius.}
Linear Algebra Appl. 234 (1996), 43--60.


\bibitem[HMoSY]{hm}  K.G. Hare, I. Morris, N. Sidorov and J. Theys \emph{An explicit counterexample to the Lagarias-Wang finiteness conjecture}. Advances in Mathematics 226 (2011) 4667--4701

\bibitem[I]{i} G.\ Iommi \emph{Ergodic optimization for renewal type shifts.} Monatsh. Math. 150 (2007), no. 2, 91--95.

%

\bibitem[IY]{iy} G.\ Iommi and Y.\ Yayama, \emph{Almost-additive thermodynamic formalism for countable Markov shifts.}
Nonlinearity 25 (2012), no. 1, 165--191.

\bibitem[J1]{j1}   O. \ Jenkinson,  \emph{Geometric Barycentres of Invariant Measures for Circle Maps.} Ergodic Theory and Dynamical Systems, 21 (2001), 511--532.

\bibitem[J2]{j2}  O. \ Jenkinson, \emph{Ergodic optimization.}  Discrete and Continuous Dynamical Systems, 15 (2006), 197--224.

\bibitem[JMU1]{jmu} O. \ Jenkinson, D.\ Mauldin and  M.\ Urb\'anski, \emph{Zero temperature limits of Gibbs-Equilibrium
states for Countable Alphabet Subshifts of Finite Type.} J. Stat. Phys. vol. 119 (2005), no. 3-4, 765-776.

\bibitem[JMU2]{jmu2}  O. \ Jenkinson, D.\ Mauldin and M.\ Urb\'anski, \emph{Ergodic optimization for non-compact dynamical systems.}  Dynamical Systems, 22 (2007), 379--388.

\bibitem[K]{k} G.\ Keller,
{\em Equilibrium states in ergodic theory.} London Mathematical Society Student Texts, 42. Cambridge University Press, Cambridge, 1998.


\bibitem[Ke]{ke} T.\ Kempton, \emph{Zero temperature limits of Gibbs equilibrium states for countable Markov shifts.} J. Stat. Phys. 143 (2011), no. 4, 795--806.

\bibitem[LW]{lw} J. C. Lagarias and Y. Wang, \emph{The finiteness conjecture for the generalized spectral radius of a set
of matrices.} Linear Algebra Appl., 214 (1995), pp. 17--42.

\bibitem[L]{l} R. Leplaideur, \emph{A dynamical proof for the convergence of Gibbs measures at temperature zero.} Nonlinearity 18 (2005), no. 6, 2847--2880.

\bibitem[MU1]{mu} D.\ Mauldin and M.\ Urb\'anski, \emph{Gibbs states on the symbolic space over an infinite alphabet.}  Israel J. Math. 125 (2001), 93--130.

\bibitem[MU2]{mu2} D.\ Mauldin and M.\ Urb\'anski, \emph{Graph directed Markov systems. Geometry and dynamics of limit sets.} Cambridge Tracts in Mathematics, 148. Cambridge University Press, Cambridge, 2003.

\bibitem[Mo]{mo} I. \ Morris \emph{Mather sets for sequences of matrices and applications to the study of joint spectral radii.} Proc. London Math. Soc. (2013) 107 (1): 121--150

\bibitem[M]{m} A.\ Mummert,  \emph{The thermodynamic formalism for almost-additive sequences.} Discrete Contin. Dyn. Syst. 16 (2006), no. 2, 435--454.

\bibitem[PP]{pp} W.\ Parry and M.\ Pollicott \emph{Zeta functions and the periodic orbit structure of hyperbolic dynamics.} Asterisque No. 187-188 (1990),
\bibitem[Pe]{pe} Y.\ Pesin,  \emph{Dimension Theory in Dynamical Systems,} CUP 1997.

\bibitem[P]{p} V. Y. Protasov, \emph{The joint spectral radius and invariant sets of linear operators.} Fundam. Prikl. Mat., 2 (1996), pp. 205--231.

\bibitem[RS]{rs} G.C. Rota and W. G. Strang, \emph{A note on the joint spectral radius.}
Indag. Math. 22 (1960), pp. 379--381


\bibitem[S1]{s1} O.\ Sarig,
{\em Thermodynamic formalism for countable Markov shifts.}
Ergodic Theory Dynam. Systems {\bf 19} (1999) 1565--1593.

\bibitem[S2]{s3} O.\ Sarig,
\emph{Existence of Gibbs measures for countable Markov shifts.}
Proc. Amer. Math. Soc. {\bf 131} (2003) 1751--1758.
%
%


\bibitem[Sc]{sc} S. Schreiber, \emph{On growth rates of subadditive functions for semiflows.}
J. Differential Equations 148 (1998), no. 2, 334--350.

\bibitem[Sch]{pablo}  P. Shmerkin \emph{Self-affine sets and the continuity of subadditive pressure.} Preprint  arXiv:1309.4730

\bibitem[SS]{ss} R. Sturman and J. Stark,
\emph{Semi-uniform ergodic theorems and applications to forced systems.}
Nonlinearity 13 (2000), no. 1, 113--143.

\bibitem[W]{w2}  P.\ Walters,  \emph{An introduction to ergodic theory.} Graduate Texts in Mathematics, 79. Springer-Verlag, New York-Berlin, (1982). ix+250 pp


\bibitem [Y]{Y} Y. \ Yayama,  \emph{Existence of a measurable saturated compensation fuction between subshifts and
its applications.} Ergodic Theory Dynam. Systems.  31 (2011), no. 5, 1563--1589.

\end{thebibliography}
\end{document}